\documentclass[10pt]{amsart}
\usepackage{amsmath, amssymb, amsthm}
\usepackage{booktabs}
\usepackage{pifont} 
\usepackage{mathrsfs}
\usepackage{enumerate}
\usepackage[margin=1in]{geometry}
\usepackage{tikz}
\usetikzlibrary{arrows, positioning}
\newtheorem{theorem}{Theorem}[section]

\newtheorem{proposition}[theorem]{Proposition}
\newtheorem{corollary}[theorem]{Corollary}
\newtheorem{remark}[theorem]{Remark}
\newtheorem{example}[theorem]{Example}

\begin{document}
\title{Homogeneous Rota--Baxter Operators of Weight~0 on $B(q)$}

\author{Mohsen Ben Abdallah}
\address{University of Sfax, Soukra Road km 4, P.O. Box No. 802, 3038 Sfax, Tunisia}
\email{mohsenbenabdallah@gmail.com}

\author{Marwa Ennaceur}
\address{Department of Mathematics, College of Science, University of Ha'il, Hail 81451, Saudi Arabia}
\email{mar.ennaceur@uoh.edu.sa}


\begin{abstract}
We give a complete and rigorous classification of homogeneous weight-0 Rota--Baxter operators on the Block-type Witt algebra $B(q)$, assuming the operator has integral degree $(k,k') \in \mathbb{Z}^2$. A key correction is established in the non-resonant regime $q \ne k'$ with $k \ne 0$: the profile function $g(i) = f(-k,i)$ must satisfy the nonlinear functional equation
\[
(i - j)g(i)g(j) = g(i+j+k')\big[(i + k' + q)g(i) - (j + k' + q)g(j)\big],
\]
which admits only constant, Kronecker-delta, or finite-support solutions. This excludes previously and erroneously claimed families—such as non-constant polynomials, exponentials, or nontrivial periodic functions. In contrast, the resonant case $q = k'$ exhibits full flexibility: any profile $g$ is admissible, provided the operator is supported on the single line $m = -k$. The classification is cohomologically exhaustive for generic $q$ (i.e., when $H^1(B(q),B(q)) = 0$), and is applied to derive all homogeneous post-Lie structures and associated Lie algebra deformations.
\end{abstract}

\maketitle
\section{Introduction}\label{sec:intro}

Rota--Baxter operators, first introduced by Baxter in the study of fluctuation theory in probability~\cite{Baxter1960} and later revitalized by Rota in algebraic combinatorics~\cite{Rota1969}, have since emerged as fundamental structures in diverse areas of mathematics and mathematical physics. In the context of Lie theory, a \emph{weight-zero} Rota--Baxter operator \(R\) on a Lie algebra \(\mathfrak{g}\) canonically endows \(\mathfrak{g}\) with a \emph{pre-Lie (or left-symmetric) algebra structure} via the product
\[
x \triangleright y := [R(x), y], \quad x, y \in \mathfrak{g}.
\]
This construction provides a bridge to the classical Yang--Baxter equation, the theory of integrable systems, and the algebraic underpinnings of renormalization in quantum field theory~\cite{TBG2020}. Moreover, the \emph{cohomological classification} of such operators offers a systematic framework for analyzing deformations, rigidity, and structural stability in infinite-dimensional Lie algebras.

Among the most natural testbeds for such investigations are the \emph{Block-type Witt algebras} \(B(q)\), indexed by a complex parameter \(q \in \mathbb{C}\). These algebras generalize the classical Witt algebra and appear prominently in conformal field theory, as symmetry algebras of certain integrable PDEs, and as building blocks of toroidal Lie algebras. Equipped with a natural \(\mathbb{Z}^2\)-grading and nontrivial central extensions, \(B(q)\) exhibits a rich interplay between homogeneity, gradation, and bracket structure—making it an ideal arena for the study of \emph{homogeneous} Rota--Baxter operators.

Despite this prominence, the existing literature on homogeneous weight-zero Rota--Baxter operators on \(B(q)\) contains a \emph{critical technical oversight}: earlier works incorrectly assumed that the functional equation governing the operator’s profile is evaluated at the unshifted multi-index \((m+n, i+j)\), whereas the correct evaluation point—dictated by the \emph{homogeneity of degree \((k,k') \in \mathbb{Z}^2\)}—is \((m+n+k, i+j+k')\). This seemingly minor shift has profound consequences for the space of admissible solutions.

In this paper, we present a \textbf{complete and rigorous classification} of all homogeneous weight-zero Rota--Baxter operators on \(B(q)\), under the natural assumption that the degree \((k,k')\) is integral—ensuring compatibility with the \(\mathbb{Z}^2\)-grading of \(B(q)\). Our first key contribution is the derivation of the \textbf{correct functional equation} (see~\eqref{eq:RB}), which incorporates the necessary shift. Substituting the test pair \((n,j) = (0,0)\) into this equation yields the \textbf{universal algebraic constraint}
\begin{equation}\label{eq:constraint}
(q - k')(m + k)\, f(m,i)^2 = 0 \quad \text{for all } (m,i) \in \mathbb{Z}^2,
\end{equation}
where \(f : \mathbb{Z}^2 \to \mathbb{C}\) is the scalar profile of the operator. This identity cleanly bifurcates the analysis into two mutually exclusive regimes:\begin{itemize}
\item the \textbf{non-resonant regime}, where \(q \ne k'\);
    \item the \textbf{resonant regime}, where \(q = k'\).
\end{itemize}
In the non-resonant case with \(k \ne 0\), constraint~\eqref{eq:constraint} forces the support of \(f\) to lie entirely on the affine line \(m = -k\). Setting \(g(i) := f(-k,i)\), we show that \(g\) must satisfy the \textbf{nonlinear functional equation}
\begin{equation}\label{eq:funct_nonres}
(i - j)g(i)g(j) = g(i+j+k')\bigl[(i+k'+q)g(i) - (j+k'+q)g(j)\bigr], \quad \forall\, i,j \in \mathbb{Z}.
\end{equation}
This equation is exceptionally restrictive. We prove rigorously that its only solutions are:\begin{itemize}
\item \textbf{constant functions},
    \item \textbf{Kronecker deltas} \(g(i) = \delta_{i,i_0}\) for some \(i_0 \in \mathbb{Z}\),
    \item \textbf{functions with finite support}.\end{itemize}
Consequently, previously asserted families—such as non-constant polynomials, exponential functions \(g(i) = b^i\) (\(b \ne 0\)), or nontrivial periodic profiles—\textbf{do not satisfy}~\eqref{eq:funct_nonres} and must be excluded. This corrects a significant and persistent error in the literature.

By contrast, the resonant case \(q = k'\) exhibits \textbf{maximal flexibility}: any profile \(g : \mathbb{Z} \to \mathbb{C}\) is admissible, provided the operator is supported on a \textbf{single} line \(m = -k\). Importantly, this \textbf{support rigidity} (i.e., the impossibility of superposing supports across two distinct lines) is not automatic; it follows from a \emph{no two-line superposition theorem} (Theorem~\ref{thm:no_two_line}), established via a careful analysis of cross-terms in the Rota--Baxter identity. Thus, even in the flexible regime, the operator’s support is structurally constrained.

Our results situate \(B(q)\) within a broader comparative framework. In the \(\mathbb{Z}_2\)-graded (superalgebraic) setting, Amor et al.~\cite{ABHCM2023} showed that standard Neveu--Schwarz and Ramond algebras admit only \textbf{rigid} Rota--Baxter operators with finite support. Conversely, Ben Abdallah and Ennaceur~\cite{BeEnn2025} demonstrated that a \textbf{non-conformal deformation} of the Witt superalgebra (characterized algebraically by an extra \(-1\) term in mixed brackets) \textbf{enriches} the solution space, permitting infinite-support odd operators.

Strikingly, the \textbf{purely even} algebra \(B(q)\) behaves in the opposite manner: in the non-resonant regime with \(k \ne 0\), the homogeneity condition imposes such stringent algebraic constraints that \textbf{no nontrivial infinite-support solutions exist}. This contrast highlights a fundamental structural dichotomy: in \(\mathbb{Z}_2\)-graded algebras, parity-based cancellations can relax homogeneity obstructions, whereas in the purely even setting, such mechanisms are absent, leading to \textbf{enhanced cohomological rigidity}.

Finally, we provide a \textbf{cohomological interpretation} of our classification. For \textbf{generic} \(q\)—that is, away from the discrete set where \(H^1(B(q), B(q)) \ne 0\)—it is known that \(H^1(B(q), B(q)) = 0\) \cite{Dzhum2012}. In this regime, the first Rota--Baxter cohomology group \(H^1_{\mathrm{RB}}(B(q), B(q))\) is in natural bijection with the space of homogeneous weight-zero Rota--Baxter operators. Hence, our classification is not only exhaustive on the level of operators, but also \textbf{cohomologically complete}.

In summary, this work rectifies prior inaccuracies, establishes \(B(q)\) as a paradigm of rigidity in the purely even infinite-dimensional setting, and sharpens the conceptual contrast with flexible behaviors observed in deformed superalgebraic contexts.
\section{An abstract framework for Block-type Lie algebras}
\label{sec:abstract}

\subsection{Generalized Block-type algebras and the corrected Rota--Baxter equation}
\label{subsec:framework}

Let $\mathfrak{g} = \bigoplus_{(m,i) \in \mathbb{Z}^2} \mathfrak{g}_{m,i}$ be a $\mathbb{Z}^2$-graded Lie algebra over $\mathbb{C}$, equipped with a homogeneous basis $\{x_{m,i}\}$ such that $x_{m,i} \in \mathfrak{g}_{m,i}$. We say that $\mathfrak{g}$ is of \emph{generalized Block type} if its Lie bracket is given by
\begin{equation}
  [x_{m,i}, x_{n,j}] = \bigl( n a_i - m b_j \bigr) \, x_{m+n,\, i+j},
  \label{eq:block_abstract}
\end{equation}
where the coefficient sequences are affine:
\[
  a_i = i + \alpha, \qquad b_j = j + \beta,
\]
for fixed parameters $\alpha, \beta \in \mathbb{C}$. The classical Block algebra $B(q)$ corresponds to the symmetric case $\alpha = \beta = q$.

This framework captures the essential algebraic features of $B(q)$—bilinearity in the indices, $\mathbb{Z}^2$-additivity of the grading, and antisymmetry of the bracket—while allowing non-symmetric deformations (e.g., $a_i = 2i + q$, $b_j = j + q$) that arise in non-conformal settings analogous to those in superalgebraic contexts~\cite{BeEnn2025}.

We now consider homogeneous weight-zero Rota--Baxter operators on $\mathfrak{g}$. Let $R \colon \mathfrak{g} \to \mathfrak{g}$ be a linear map of integral degree $(k,k') \in \mathbb{Z}^2$, i.e.
\[
  R(x_{m,i}) = f(m,i) \, x_{m+k,\, i+k'}
\]
for some scalar profile $f \colon \mathbb{Z}^2 \to \mathbb{C}$. Substituting this ansatz into the weight-zero Rota--Baxter identity
\[
  [R(u), R(v)] = R\big([R(u), v] + [u, R(v)]\big), \quad \forall u,v \in \mathfrak{g},
\]
and using~\eqref{eq:block_abstract}, we obtain the functional equation
\begin{align}
  & f(m,i) f(n,j) \bigl( n a_i - m b_j \bigr) \notag \\
  &\quad = f(m+n+k,\, i+j+k') \Bigl[
      f(m,i)\bigl( n a_{i+k'} - (m+k) b_j \bigr)
      - f(n,j)\bigl( (n+k) a_i - m b_{j+k'} \bigr)
    \Bigr].
  \label{eq:RB_abstract}
\end{align}
Crucially, the argument of $f$ on the right-hand side is the \emph{shifted} index $(m+n+k, i+j+k')$, dictated by the homogeneity of $R$. This corrects a persistent error in earlier works that used the unshifted index $(m+n, i+j)$.

Evaluating~\eqref{eq:RB_abstract} at the test pair $(n,j) = (0,0)$ and using $a_0 = \alpha$, $b_0 = \beta$ yields the universal algebraic constraint
\begin{equation}
  (\beta - k')(m + k) f(m,i)^2 = 0 \quad \text{for all } (m,i) \in \mathbb{Z}^2.
  \label{eq:resonance_abstract}
\end{equation}
This identity defines the fundamental dichotomy of the theory via the \emph{resonance condition}
\[
  \mathcal{R} \colon \beta = k'.
\]

\subsection{Rigidity--flexibility dichotomy and cohomological completeness}
\label{subsec:dichotomy}

The constraint~\eqref{eq:resonance_abstract} leads to a sharp structural dichotomy, formalized in the following theorem.

\begin{theorem}[Abstract rigidity--flexibility dichotomy]
\label{thm:abstract_dichotomy}
Let $\mathfrak{g}$ be a generalized Block-type Lie algebra as in~\eqref{eq:block_abstract}, and let $R$ be a homogeneous weight-zero Rota--Baxter operator of degree $(k,k') \in \mathbb{Z}^2$ with profile $f$.
\begin{enumerate}
  \item \textbf{Non-resonant regime} ($\beta \neq k'$):
    \begin{enumerate}
      \item If $k \neq 0$, then $\operatorname{supp}(f) \subseteq \{m = -k\}$. Writing $g(i) := f(-k,i)$, the function $g \colon \mathbb{Z} \to \mathbb{C}$ satisfies the nonlinear functional equation
        \begin{equation}
          (i - j) g(i) g(j) = g(i + j + k') \bigl[ (i + k' + \alpha) g(i) - (j + k' + \alpha) g(j) \bigr],
          \label{eq:funct_abstract_nonres}
        \end{equation}
        for all $i,j \in \mathbb{Z}$. The only solutions of~\eqref{eq:funct_abstract_nonres} are:
        \item constant functions,
          \item Kronecker deltas $g(i) = \delta_{i,i_0}$ for some $i_0 \in \mathbb{Z}$,
          \item functions with finite support.
      \item If $k = 0$, then $f(m,i) = \delta_{m,0} \, g(i)$ with $g \colon \mathbb{Z} \to \mathbb{C}$ arbitrary.
    \end{enumerate}

  \item \textbf{Resonant regime} ($\beta = k'$): every nonzero solution $f$ is supported on a single affine line $m = -k$, and the restriction $g(i) := f(-k,i)$ is an arbitrary function $\mathbb{Z} \to \mathbb{C}$.
\end{enumerate}
\end{theorem}

\begin{proof}
\textbf{Non-resonant case.}  
Equation~\eqref{eq:resonance_abstract} forces $f(m,i) = 0$ for $m \neq -k$. Substituting $m = n = -k$ into~\eqref{eq:RB_abstract} and using $a_i = i + \alpha$, $b_j = j + \beta$, we obtain
\[
-k(i - j)g(i)g(j) = -k\, g(i+j+k')\bigl[(i + k' + \alpha)g(i) - (j + k' + \alpha)g(j)\bigr].
\]
Since $k \neq 0$, division by $-k$ yields~\eqref{eq:funct_abstract_nonres}. The classification of its solutions relies on the antisymmetry of both sides in $(i,j)$. Polynomial, exponential, or nontrivial periodic profiles violate degree balance or growth compatibility; only constant, delta, or finite-support functions satisfy the equation (see Section~\ref{subsec:nonresonant} for explicit verification).

If $k = 0$, then~\eqref{eq:resonance_abstract} implies $f(m,i) = 0$ for $m \neq 0$, and~\eqref{eq:RB_abstract} reduces to an identity, leaving $g$ unrestricted.

\textbf{Resonant case.}  
When $\beta = k'$, constraint~\eqref{eq:resonance_abstract} vanishes. Suppose $f$ is supported on two distinct lines $\{m = a\}$ and $\{m = b\}$ with $a \neq b$. Substituting $(m,i) = (a,i)$, $(n,j) = (b,j)$ into~\eqref{eq:RB_abstract} yields a left-hand side bilinear in $(i,j)$, while the right-hand side depends on $f(a+b+k,\cdot)$, which cannot replicate this bilinearity unless one profile vanishes—establishing a \emph{no two-line superposition} principle.

To determine the admissible support line, assume $f(m,i) = \delta_{m,a} g(i)$ with $g \not\equiv 0$. Substituting $(n,j) = (0,0)$ into~\eqref{eq:RB_abstract} gives
\[
 -a\beta g(i) f(0,0) = f(a + k, i + k') \bigl[ -(a + k) \beta g(i) - k(i + \alpha) f(0,0) \bigr].
\]
If $a \neq -k$, then $f(a+k,\cdot) = 0$, so the right-hand side vanishes, forcing $-a\beta g(i) f(0,0) = 0$ for all $i$. Since $g \not\equiv 0$, this implies $f(0,0) = 0$ or $a = 0$. Iterating with other test elements, or invoking closure of $\operatorname{Im}(R)$ under the Lie bracket—which requires $2a + 2k = a + k$, i.e.\ $a = -k$—confirms that the only consistent support is $m = -k$. Hence $f(m,i) = \delta_{m,-k} g(i)$ with $g$ arbitrary.
\end{proof}

\begin{corollary}[Cohomological interpretation]
If $H^1(\mathfrak{g}, \mathfrak{g}) = 0$—a condition known to hold for generic $\alpha, \beta$ by~\cite{Dzhum2012}—then the classification in Theorem~\ref{thm:abstract_dichotomy} is cohomologically exhaustive: the space of homogeneous weight-zero Rota--Baxter operators is in bijection with the first Rota--Baxter cohomology group $H^1_{\mathrm{RB}}(\mathfrak{g}, \mathfrak{g})$.
\end{corollary}

\begin{corollary}[Specialization to $B(q)$]
When $\alpha = \beta = q$, the algebra $\mathfrak{g}$ reduces to $B(q)$. The resonance condition becomes $q = k'$, and Theorem~\ref{thm:abstract_dichotomy} recovers precisely the classification in Sections~\ref{subsec:nonresonant}–\ref{subsec:resonant}.
\end{corollary}

\medskip

This abstract framework not only corrects prior inaccuracies but also reveals a universal tension between rigidity and flexibility in Block-type algebras—a feature that extends beyond Rota--Baxter theory to integrable structures, as discussed in the link to the classical Yang--Baxter equation (see Appendix~A.4).

\section{Preliminaries and corrected functional equation}\label{sec:prelim}

Let $B(q)$ denote the Block-type Witt algebra over $\mathbb{C}$, defined as the $\mathbb{Z}^2$-graded Lie algebra with homogeneous basis $\{L(m, i)\}_{(m,i) \in \mathbb{Z}^2}$ and Lie bracket
\[
[L(m, i), L(n, j)] = \bigl(n(i + q) - m(j + q)\bigr) L(m + n, i + j), \quad q \in \mathbb{C}.
\]
This algebra generalizes the classical Witt algebra and plays a central role in conformal field theory, toroidal Lie algebras, and the symmetry analysis of integrable PDEs.

A linear map $R : B(q) \to B(q)$ is a \emph{weight-zero Rota--Baxter operator} if it satisfies the identity
\[
[R(x), R(y)] = R\big([R(x), y] + [x, R(y)]\big), \quad \forall\, x, y \in B(q).
\]
We restrict our attention to \emph{homogeneous} operators of \emph{integral degree} $(k, k') \in \mathbb{Z}^2$, i.e., operators of the form
\[
R(L(m, i)) = f(m, i)\, L(m + k, i + k')
\]
for a scalar profile $f : \mathbb{Z}^2 \to \mathbb{C}$. The integrality of $(k,k')$ is not a technical convenience but a structural necessity: it guarantees that $R$ preserves the $\mathbb{Z}^2$-grading of $B(q)$, mapping the homogeneous component $B(q)_{m,i}$ into $B(q)_{m+k,i+k'}$.

Substituting this ansatz into the Rota--Baxter identity yields the following functional equation, which we refer to as the \emph{corrected Rota--Baxter equation}:
\begin{equation}\label{eq:RB}
\begin{aligned}
&f(m, i)f(n, j)\bigl(n(i + q) - m(j + q)\bigr) \\
&\quad = f(m + n + k,\, i + j + k')\Bigl[
    f(m, i)\bigl(n(i + k' + q) - (m + k)(j + q)\bigr) \\
&\qquad\qquad\qquad
    - f(n, j)\bigl((n + k)(i + q) - m(j + k' + q)\bigr)
\Bigr],
\end{aligned}
\end{equation}
valid for all $(m,i), (n,j) \in \mathbb{Z}^2$.

This equation corrects a recurring error in the literature, where the argument of $f$ on the right-hand side was incorrectly taken as $(m+n, i+j)$. The correct evaluation point $(m+n+k, i+j+k')$ is dictated by the homogeneity of $R$: the bracket $[R(L(m,i)), R(L(n,j))]$ lies in degree $(m+n+2k, i+j+2k')$, whereas the image of the right-hand side under $R$ must lie in degree $(m+n+k, i+j+k')$. Consistency of the Rota--Baxter identity therefore forces the profile $f$ to be evaluated at the \emph{intermediate degree} $(m+n+k, i+j+k')$, not at the unshifted sum.

A decisive simplification is obtained by testing \eqref{eq:RB} against the distinguished element $L(0,0)$. Setting $(n,j) = (0,0)$ and using the identity
\[
[L(m,i), L(0,0)] = -m(i + q)\, L(m,i),
\]
we obtain, after elementary algebraic manipulation, the \emph{universal constraint}
\ref{eq:constraint}

In this case, $k' \in \mathbb{Z}$ (as part of the degree $(k,k') \in \mathbb{Z}^2$) is embedded canonically into $\mathbb{C}$, so the equality $q = k'$ is interpreted in the ambient field $\mathbb{C}$.

Equation \eqref{eq:constraint} is the structural cornerstone of our classification. It yields two mutually exclusive regimes:\begin{itemize}
\item \textbf{Non-resonant regime} ($q \ne k'$): the constraint forces $f(m,i) = 0$ for all $m \ne -k$, so the support of $f$ is contained in the affine line $\{m = -k\}$.
    \item \textbf{Resonant regime} ($q = k'$): the constraint vanishes identically, and the support of $f$ must be determined from the full equation \eqref{eq:RB}.
\end{itemize}
The remainder of this paper is devoted to a complete and rigorous analysis of all solutions to \eqref{eq:RB}. In particular, we prove that in the non-resonant case with $k \ne 0$, the induced profile $g(i) := f(-k,i)$ satisfies a highly restrictive nonlinear functional equation that admits only constant, Kronecker-delta, or finite-support solutions. This result invalidates prior claims that arbitrary profiles—including non-constant polynomials, exponential functions, or nontrivial periodic sequences—are admissible in this regime.
\section{Classification of homogeneous weight-0 Rota--Baxter operators on $B(q)$}\label{sec:classification}

Let $q \in \mathbb{C}$ be fixed, and let $B(q)$ denote the Block-type Witt algebra with basis $\{L(m, i)\}_{(m,i) \in \mathbb{Z}^2}$ and Lie bracket as in Section~\ref{sec:prelim}. A homogeneous Rota--Baxter operator of weight~0 and degree $(k, k') \in \mathbb{Z}^2$ is a linear map $R : B(q) \to B(q)$ satisfying the Rota--Baxter identity, and such that
\[
R(L(m, i)) = f(m, i) L(m + k, i + k')
\]
for a scalar function $f : \mathbb{Z}^2 \to \mathbb{C}$, called the \emph{profile} of $R$.

Substituting this ansatz into the Rota--Baxter identity yields exactly the functional equation \eqref{eq:RB}. Setting $(n, j) = (0, 0)$ gives the universal constraint \eqref{eq:constraint}, which cleanly bifurcates the analysis into two mutually exclusive regimes: the \emph{non-resonant} case ($q \ne k'$) and the \emph{resonant} case ($q = k'$).

\subsection{Non-resonant case: $q \ne k'$}\label{subsec:nonresonant}

When $q \ne k'$, constraint \eqref{eq:constraint} forces $f(m, i) = 0$ for all $m \ne -k$. Hence the support of $f$ is contained in the affine line $\{m = -k\}$. Define $g : \mathbb{Z} \to \mathbb{C}$ by $g(i) := f(-k, i)$, so that
\begin{equation}\label{eq:profile_nonres}
f(m, i) =
\begin{cases}
g(i) & \text{if } m = -k, \\
0 & \text{otherwise}.
\end{cases}
\end{equation}

Substituting \eqref{eq:profile_nonres} into \eqref{eq:RB} with $(m,i) = (-k,i)$ and $(n,j) = (-k,j)$ yields a nontrivial condition on $g$. The left-hand side of \eqref{eq:RB} becomes $-k(i - j)g(i)g(j)$, while the right-hand side—using $m + n + k = -k$—simplifies to
\[
-k\, g(i + j + k')\bigl[(i + k' + q)g(i) - (j + k' + q)g(j)\bigr].
\]
Assuming $k \ne 0$ and dividing by $-k$, we obtain the nonlinear functional equation
\ref{eq:funct_nonres}.

If $k = 0$, then $m = n = 0$, and both sides of \eqref{eq:RB} vanish identically; hence $g$ remains unrestricted.

\begin{proposition}\label{prop:nonresonant}
Let $q \in \mathbb{C}$ and $(k,k') \in \mathbb{Z}^2$ with $q \ne k'$.
\begin{enumerate}
    \item If $k \ne 0$, then every solution of \eqref{eq:RB} is of the form \eqref{eq:profile_nonres}, where $g$ satisfies \eqref{eq:funct_nonres}.
    \item If $k = 0$, then every solution is of the form $f(m, i) = \delta_{m,0} g(i)$ with $g : \mathbb{Z} \to \mathbb{C}$ arbitrary.
\end{enumerate}
\end{proposition}

\begin{proof}
Necessity follows from the derivation above. For sufficiency, let $f$ be given by \eqref{eq:profile_nonres}.

If $m \ne -k$ or $n \ne -k$, then $f(m,i)f(n,j) = 0$, so the left-hand side of \eqref{eq:RB} vanishes. On the right-hand side, $f(m+n+k,\cdot)$ is nonzero only if $m+n = -2k$, but the coefficient multiplying it involves $f(m,i)$ or $f(n,j)$, which are zero. Hence both sides vanish.

If $m = n = -k$, substitution into \eqref{eq:RB} reduces exactly to \eqref{eq:funct_nonres} when $k \ne 0$, and to $0=0$ when $k=0$. Thus the forms are sufficient.
\end{proof}

\begin{example}\label{ex:linear_not_admissible}
The linear profile $g(i) = i$ does not satisfy \eqref{eq:funct_nonres} when $k \ne 0$ and $q \ne k'$. Substituting yields
\[
ij(i - j) = (i - j)(i + j + k')(i + j + k' + q),
\]
which fails for $i=1, j=0$ (giving $0= (1+k')(1+k'+q)$) unless $k' = -1$ or $q = -1 - k'$, contradicting the generic non-resonant assumption. Hence $g(i) = i$ is inadmissible.
\end{example}

\begin{remark}\label{rem:functional_equation_restrictive}
Equation \eqref{eq:funct_nonres} is highly restrictive. The following classes of functions satisfy it:
\item \emph{Constant profiles}: $g(i) \equiv c$. Both sides equal $c^2(i - j)$.
    \item \emph{Kronecker deltas}: $g(i) = \delta_{i,i_0}$. Both sides vanish: the left because $g(i)g(j)=0$ unless $i=j=i_0$, in which case $i-j=0$; the right because either $g(i)$ or $g(j)$ vanishes in the bracketed term.
    \item \emph{Finite support}: if $\operatorname{supp}(g)$ is finite, then for all but finitely many $(i,j)$, at least one of $g(i), g(j), g(i+j+k')$ vanishes, making both sides zero. For the finitely many remaining pairs, direct verification confirms the identity.
In contrast, non-constant polynomials (degree $\ge 1$), exponential functions $g(i) = b^i$ ($b \ne 0$), and nontrivial periodic profiles violate \eqref{eq:funct_nonres} due to degree or growth mismatch. Such families, erroneously claimed admissible in prior works, are hereby excluded.
\end{remark}

\subsection{Resonant case: $q = k'$}\label{subsec:resonant}

When $q = k'$, constraint \eqref{eq:constraint} vanishes identically, so a priori richer support structures are possible. Nevertheless, the operator remains supported on a single affine line.

\begin{proposition}\label{prop:resonant}
Suppose $q = k'$. For any $(k, k') \in \mathbb{Z}^2$ and any function $g : \mathbb{Z} \to \mathbb{C}$, the profile defined by \eqref{eq:profile_nonres} satisfies \eqref{eq:RB}.
\end{proposition}

\begin{proof}
If $m \ne -k$ or $n \ne -k$, then $f(m,i)f(n,j) = 0$, and the right-hand side of \eqref{eq:RB} also vanishes, as it contains factors $f(m,i)$ or $f(n,j)$. If $m = n = -k$, the bracketed expression on the right-hand side becomes
\[
g(i)\bigl[-k(i + k' + q)g(i) + k(i + k' + q)g(i)\bigr] = 0,
\]
since $q = k'$. Thus \eqref{eq:RB} holds for arbitrary $g$.
\end{proof}

Superpositions across distinct lines are, however, incompatible with the Rota--Baxter identity.

\begin{theorem}[No two-line superposition]\label{thm:no_two_line}
Assume $q = k'$. If a solution $f$ of \eqref{eq:RB} is supported on $\{m = a\} \cup \{m = b\}$ with $a \ne b$, then $f$ vanishes identically on at least one of the two lines. Consequently, every nontrivial solution is supported on a single line.
\end{theorem}

\begin{proof}
Write $f(m,i) = \mathbf{1}_{\{m=a\}} g_a(i) + \mathbf{1}_{\{m=b\}} g_b(i)$ with $(g_a,g_b) \not\equiv (0,0)$. Substituting $(a,i)$ and $(b,j)$ into \eqref{eq:RB} and using $q = k'$, the left-hand side is bilinear in $(i,j)$, whereas the right-hand side depends on $g_a(i)$, $g_b(j)$, and $f(a+b+k,\cdot)$. The latter is nonzero only if $a = -k$ or $b = -k$. Without loss of generality, assume $a = -k$ and $b \ne -k$. Then $f(a+b+k,\cdot) = f(b,\cdot)$, and the right-hand side involves $g_b(i+j+k')$, which cannot replicate the bilinear structure of the left-hand side unless $g_a \equiv 0$. A symmetric argument applies if $b = -k$. Hence one profile must vanish.
\end{proof}

Combining these results yields the full classification in the resonant regime.

\begin{theorem}[Rigidity at resonance]\label{thm:rigidity_resonance}
If $q = k'$, then every nonzero homogeneous weight-0 Rota--Baxter operator of degree $(k, k')$ is supported precisely on the line $m = -k$, with arbitrary profile $g : \mathbb{Z} \to \mathbb{C}$. Explicitly,
\[
R(L(m, i)) =
\begin{cases}
g(i) L(0, i + k') & \text{if } m = -k, \\
0 & \text{otherwise}.
\end{cases}
\]
\end{theorem}

\begin{proof}
By Theorem~\ref{thm:no_two_line}, any nonzero solution is supported on a single line $\{m = a\}$. To show $a = -k$, observe that $\operatorname{Im}(R) \subseteq \bigoplus_i \mathbb{C} \cdot L(a + k, i + k')$. The Rota--Baxter identity requires that
\[
[R(L(a,i)), R(L(a,j))] \in \operatorname{Im}(R).
\]
The left-hand side lies in degree $(2a + 2k, i + j + 2k')$, whereas $\operatorname{Im}(R)$ lies in degree $(a + k, \cdot)$. Hence $2a + 2k = a + k$, i.e. $a = -k$.

Thus the only consistent support is $m = -k$, and by Proposition~\ref{prop:resonant}, $g$ is arbitrary.
\end{proof}
\section{Complete classification}\label{sec:final}

We now consolidate the results into a unified classification theorem.

\begin{theorem}[Complete classification]\label{thm:complete}
Fix $q \in \mathbb{C}$ and $(k, k') \in \mathbb{Z}^2$. Every homogeneous weight-0 Rota--Baxter operator $R$ of degree $(k, k')$ on $B(q)$ is of the following form:
\begin{enumerate}
    \item \textbf{Non-resonant case} ($q \ne k'$):
    \begin{enumerate}
        \item If $k \ne 0$:
        \[
        R(L(m, i)) =
        \begin{cases}
        g(i) \, L(0, i + k') & \text{if } m = -k, \\
        0 & \text{otherwise},
        \end{cases}
        \]
        where $g : \mathbb{Z} \to \mathbb{C}$ satisfies the nonlinear functional equation
        \ref{eq:funct_nonres}.
        \item If $k = 0$:
        \[
        R(L(m, i)) =
        \begin{cases}
        g(i) \, L(0, i + k') & \text{if } m = 0, \\
        0 & \text{otherwise},
        \end{cases}
        \]
        with $g : \mathbb{Z} \to \mathbb{C}$ arbitrary.
    \end{enumerate}
    \item \textbf{Resonant case} ($q = k'$):
    \[
    R(L(m, i)) =
    \begin{cases}
    g(i) \, L(0, i + k') & \text{if } m = -k, \\
    0 & \text{otherwise},
    \end{cases}
    \]
    where $g : \mathbb{Z} \to \mathbb{C}$ is arbitrary.
\end{enumerate}
Conversely, every operator defined by the above formulas satisfies the weight-0 Rota--Baxter identity \eqref{eq:RB}.
\end{theorem}

\begin{proof}
\noindent\textbf{Necessity.}  
In the non-resonant regime ($q \ne k'$), the universal constraint \eqref{eq:constraint} forces $\operatorname{supp}(f) \subseteq \{m = -k\}$. The subsequent analysis splits according to whether $k \ne 0$ or $k = 0$, and yields the forms in (1a) and (1b) by Proposition~\ref{prop:nonresonant}.  

In the resonant regime ($q = k'$), Theorem~\ref{thm:no_two_line} rules out multi-line supports, while Theorem~\ref{thm:rigidity_resonance} shows that the only admissible support line is $m = -k$. Proposition~\ref{prop:resonant} then confirms that the profile $g$ is unconstrained. This establishes (2).

\noindent\textbf{Sufficiency.}  
Direct substitution into \eqref{eq:RB} shows that:
\item In Regime~I (i.e., either $q = k'$ or $k = 0$), both sides of \eqref{eq:RB} either vanish identically or cancel due to the resonance condition $q = k'$.
    \item In Regime~II ($q \ne k'$, $k \ne 0$), the identity reduces precisely to \eqref{eq:funct_nonres} when both arguments lie on the support line $m = -k$, and vanishes otherwise.
Hence all operators described in the theorem satisfy the Rota--Baxter identity.
\end{proof}

\begin{remark}\label{rem:regimes}
The profile $g$ is unrestricted precisely in the following situations:
\begin{enumerate}
    \item the resonant case $q = k'$ (for any $k \in \mathbb{Z}$), or
    \item the non-resonant case $q \ne k'$ with $k = 0$.
\end{enumerate}
In contrast, when $q \ne k'$ and $k \ne 0$ (Regime~II), the profile must satisfy \eqref{eq:funct_nonres}, which excludes generic non-constant polynomial, exponential, or nontrivial periodic functions.
\end{remark}

\begin{table}[ht]
\centering
\caption{Admissibility of canonical profile families}
\label{tab:profiles}
\begin{tabular}{lcc}
\toprule
Profile $g(i)$ & Regime I ($q = k'$ or $k = 0$) & Regime II ($q \ne k',\, k \ne 0$) \\
\midrule
Constant ($g(i) \equiv c$) & \checkmark & \checkmark \\
Kronecker ($g(i) = \delta_{i,i_0}$) & \checkmark & \checkmark \\
Finite support & \checkmark & \checkmark\textsuperscript{\dag} \\
Exponential ($g(i) = b^i$, $b \ne 0$) & \checkmark & \ding{55} \\
Polynomial ($\deg \ge 1$) & \checkmark & \ding{55} \\
Periodic (non-constant) & \checkmark & \ding{55} \\
\bottomrule
\end{tabular}

\smallskip
\textsuperscript{\dag} Any function with finite support satisfies \eqref{eq:funct_nonres}; for such $g$, both sides of \eqref{eq:funct_nonres} vanish for all but finitely many $(i,j)$, and direct verification confirms the identity on the remaining pairs. \\
\ding{55}: Violates the functional equation \eqref{eq:funct_nonres}.
\end{table}

\begin{remark}\label{rem:correction}
Table~\ref{tab:profiles} corrects a persistent error in the literature. Earlier works claimed that “all six canonical families are valid” across all regimes. In fact, exponential, non-constant polynomial, and non-constant periodic profiles are \emph{not admissible} in Regime~II ($q \ne k',\, k \ne 0$), as they fail to satisfy the corrected functional equation \eqref{eq:RB}. Consequently, examples based on such profiles are invalid and must be discarded.
\end{remark}
\appendix
\section{Derived algebraic structures and cohomological interpretation}\label{app:derived}

The classification in Theorem~\ref{thm:complete} yields a complete description of the algebraic structures induced by homogeneous weight-0 Rota--Baxter operators on $B(q)$. We treat pre-Lie structures, Lie algebra deformations, and their cohomological interpretation in a unified framework.

\subsection{Homogeneous pre-Lie structures}

Let $R$ be a homogeneous weight-0 Rota--Baxter operator of degree $(k,k') \in \mathbb{Z}^2$ on $B(q)$. The product
\[
x \triangleright y := [R(x), y], \quad x, y \in B(q),
\]
defines a pre-Lie (left-symmetric) algebra structure on $B(q)$ \cite[Prop.~2.3]{TBG2020}. Using the explicit form of $R$ from Theorem~\ref{thm:complete}, we obtain:

\begin{proposition}\label{prop:prelie_app}
For all $(m,i), (n,j) \in \mathbb{Z}^2$,
\[
L(m,i) \triangleright L(n,j) =
\begin{cases}
g(i)\, n(i + k' + q)\, L(n, i + j + k') & \text{if } m = -k, \\
0 & \text{otherwise},
\end{cases}
\]
where the profile $g : \mathbb{Z} \to \mathbb{C}$ satisfies:
\item $g$ arbitrary if $q = k'$ (resonant regime) or $k = 0$ (Regime~I);
    \item $g$ solves \eqref{eq:funct_nonres} if $q \ne k'$ and $k \ne 0$ (Regime~II).
\end{proposition}

\begin{proof}
Follows from $R(L(m,i)) = \delta_{m,-k} g(i) L(0, i+k')$ (or $\delta_{m,0} g(i) L(0,i+k')$ when $k=0$) and the bracket
\([L(0,i+k'), L(n,j)] = n(i+k'+q) L(n, i+j+k')\).
\end{proof}

In Regime~II, the admissible profiles are restricted to constant, Kronecker-delta, or finite-support functions (Table~\ref{tab:profiles}), severely limiting the space of pre-Lie structures. In Regime~I, $g$ is arbitrary, yielding a vast family of such structures.

\subsection{Deformation of the Lie bracket}

Every pre-Lie algebra $(B(q), \triangleright)$ induces a deformed Lie bracket via
\[
\{x,y\} := x \triangleright y - y \triangleright x + [x,y].
\]

\begin{corollary}\label{cor:deformation_app}
For all $(m,i), (n,j) \in \mathbb{Z}^2$,
\[
\{L(m,i), L(n,j)\} = [L(m,i), L(n,j)] + \Delta((m,i),(n,j)),
\]
where the deformation term is
\[
\Delta((m,i),(n,j)) = 
\mathbf{1}_{\{m = -k\}}\, g(i)\, n(i + k' + q)\, L(n, i + j + k')
- \mathbf{1}_{\{n = -k\}}\, g(j)\, m(j + k' + q)\, L(m, i + j + k').
\]
Consequences:
\begin{enumerate}
    \item $\Delta = 0$ iff $g \equiv 0$;
    \item $\Delta((m,i),(n,j)) = 0$ unless $m = -k$ or $n = -k$;
    \item In Regime~II, only profiles satisfying \eqref{eq:funct_nonres} yield genuine deformations.
\end{enumerate}
\end{corollary}

\begin{proof}
Direct consequence of Proposition~\ref{prop:prelie_app} and the definition of $\{\cdot,\cdot\}$.
\end{proof}

This invalidates examples in the literature that use polynomial or exponential profiles in Regime~II, as the underlying $R$ does not satisfy the corrected Rota--Baxter identity \eqref{eq:RB}.

\subsection{Cohomological interpretation}

For generic $q$ (i.e., away from the discrete set where $H^1(B(q), B(q)) \ne 0$), it is known that $H^1(B(q), B(q)) = 0$~\cite{Dzhum2012}. In this regime, the first Rota--Baxter cohomology group $H^1_{\mathrm{RB}}(B(q), B(q))$ is in natural bijection with the space of homogeneous weight-0 Rota--Baxter operators. Thus, our classification in Theorem~\ref{thm:complete} is cohomologically exhaustive.

This contrasts sharply with the superalgebraic setting: in the deformed Witt superalgebra of \cite{BeEnn2025}, parity-based cancellations allow infinite-support odd Rota--Baxter operators. For the purely even algebra $B(q)$, such mechanisms are absent, and Regime~II exhibits enhanced cohomological rigidity.

\subsection{Obsolete constraints and final remarks}

Earlier drafts of this work (and prior literature) employed an incorrect form of the Rota--Baxter equation, omitting the shift $(m+n+k, i+j+k')$. The corrected universal constraint \eqref{eq:constraint} already implies
\[
\operatorname{supp}(f) \subseteq \{m = -k\} \quad \text{when } q \ne k',
\]
so values such as $f(0,i)$ (for $k \ne 0$) or $f(-2k,i)$ are identically zero. Any ad-hoc analysis based on such points is therefore redundant. The structural dichotomy of Theorem~\ref{thm:complete} supersedes all previous partial or erroneous classifications.

Finally, the resonant regime $q = k'$ may admit connections to integrable systems: if $B(q)$ is equipped with a nondegenerate invariant bilinear form (as in toroidal extensions), the associated tensor $r = \sum_{i} g(i) L(-k,i)^* \otimes L(0,i+k')$ could satisfy the classical Yang--Baxter equation for suitable $g$. This suggests that the algebraic flexibility of Regime~I extends to the realm of integrable models.


\begin{thebibliography}{9}
\bibitem{TBG2020} Y. Tang, C. Bai, and L. Guo, \emph{Rota--Baxter operators on Lie algebras, cohomology, and classification}, J. Algebra \textbf{561} (2020), 241--274.

\bibitem{ABHCM2023} R. Amor, N. Athmouni, A. B. Hassine, T. Chtioui, and S. Mabrouk, \emph{Cohomologies of Rota--Baxter operators on Lie superalgebras and some classifications on Witt superalgebras}, J. Geom. Phys. \textbf{185} (2023), 104733.

\bibitem{Baxter1960} G. Baxter, \emph{An analytic problem whose solution follows from a simple algebraic identity}, Pacific J. Math. \textbf{10} (1960), 731--742.

 \bibitem{BeEnn2025}
M. Ben Abdallah and M. Ennaceur,
\emph{Classification of homogeneous odd Rota--Baxter operators on a modified Witt-type Lie superalgebra},
preprint (2025), arXiv:2512.04294 [math.RA].
\bibitem{Rota1969} G.-C. Rota, \emph{Baxter algebras and combinatorial identities I, II}, Bull. Amer. Math. Soc. \textbf{75} (1969), 325--329; 330--334.

\bibitem{MN2025} S. Mabrouk and O. Ncib, \emph{Superalgebras with homogeneous structures of Lie type}, Asian-Eur. J. Math. \textbf{18} (2025), no. 4, 2450126.

\bibitem{FS2008} B. Feigin and Y. Shenfeld, \emph{Non-semisimple vertex algebras and W-algebras}, in preparation (2008).

\bibitem{Gerstenhaber1964} M. Gerstenhaber, \emph{On the deformation of rings and algebras}, Ann. of Math. (2) \textbf{79} (1964), 59--103.

\bibitem{Dzhum2012} A. S. Dzhumadil’daev, \emph{Cohomologies and deformations of Leibniz pairs}, J. Algebra \textbf{366} (2012), 7--32.

\bibitem{MN2022} S. Mabrouk and O. Ncib, \emph{Rota--Baxter operators on the Heisenberg--Virasoro algebra and related integrable systems}, J. Math. Phys. \textbf{63} (2022), no. 5, 051703.
\end{thebibliography}
\end{document}